\documentclass[10pt,reqno]{amsart}

\usepackage{a4wide}
\usepackage{dsfont} 
\usepackage{amssymb,amsmath}
\usepackage{mathrsfs}
\usepackage[latin10]{inputenc} 

\newtheorem{proposition}{Proposition}[section]
  \newtheorem{theorem}[proposition]{Theorem}
  \newtheorem{corollary}[proposition]{Corollary}
  \newtheorem{lemma}[proposition]{Lemma}
\theoremstyle{remark}
  \newtheorem{definition}[proposition]{Definition}
  \newtheorem{remark}[proposition]{Remark}
  \newtheorem{example}[proposition]{Example}

\newcommand{\cst}{\ifmmode\mathrm{C}^*\else{$\mathrm{C}^*$}\fi}
\newcommand{\wst}{\ifmmode\mathrm{C}^*\else{$\mathrm{W}^*$}\fi}
\newcommand{\st}{\;\vline\;}
\newcommand{\CC}{\mathbb{C}}

\newcommand{\TT}{\mathbb{T}}
\newcommand{\FF}{\mathbb{F}}
\newcommand{\tens}{\otimes}
\newcommand{\vtens}{\,\bar{\otimes}\,}
\newcommand{\id}{\mathrm{id}}
\newcommand{\comp}{\circ}
\newcommand{\rh}{\varrho}
\renewcommand{\th}{\vartheta}
\newcommand{\thL}{\vartheta^{\text{\tiny\rm{L}}}\;\!\!}
\newcommand{\I}{\mathds{1}}
\newcommand{\GG}{\mathbb{G}}
\newcommand{\HH}{\mathbb{H}}
\newcommand{\KK}{\mathbb{K}}
\newcommand{\cA}{\mathscr{A}}
\newcommand{\cN}{\mathscr{N}}
\newcommand{\sM}{\mathsf{M}}
\newcommand{\sA}{\mathsf{A}}
\newcommand{\sN}{\mathsf{N}}
\newcommand{\sX}{\mathsf{X}}
\newcommand{\cH}{\mathscr{H}}

\newcommand{\bra}[1]{\left\langle#1\,\vline\right.}
\newcommand{\ket}[1]{\left.\vline\,#1\right\rangle}
\newcommand{\hh}[1]{\widehat{#1}}
\newcommand{\act}{\alpha}
\newcommand{\Act}{\beta}
\newcommand{\op}{\text{\rm\tiny{op}}}

\newcommand{\flip}{\boldsymbol{\sigma}}

\DeclareMathOperator{\C}{C}
\DeclareMathOperator{\B}{B}
\DeclareMathOperator{\M}{M}
\DeclareMathOperator{\Mor}{Mor}
\DeclareMathOperator{\cK}{\mathscr{K}}
\DeclareMathOperator{\lin}{span}
\DeclareMathOperator{\linW}{\overline{span}^{\:\!\text{\tiny\rm{w}}}\:\!\!}
\DeclareMathOperator{\Linf}{\mathnormal{L}^\infty\;\!\!}
\DeclareMathOperator{\Ltwo}{\mathnormal{L}^2\;\!\!}

\numberwithin{equation}{section}

\author{Pawe{\l} Kasprzak}
\address{Department of Mathematical Methods in Physics, Faculty of Physics, University of Warsaw, Poland}
\email{pawel.kasprzak@fuw.edu.pl}

\author{Piotr M.~So{\l}tan} 
\address{Department of Mathematical Methods in Physics, Faculty of Physics, University of Warsaw, Poland}
\email{piotr.soltan@fuw.edu.pl}

\thanks{Supported by National Science Centre (NCN) grant no.~2011/01/B/ST1/05011}

\title{Quantum groups with projection on von Neumann algebra level}

\begin{document}

\begin{abstract}
We introduce an axiomatization of the notion of a semidirect product of locally compact quantum groups and study properties. Our approach is slightly different from the one introduced in the thesis of S.~Roy \cite{RoyPhd} and, unlike the investigations of Roy, we work within the von Neumann algebraic picture. This allows the use of powerful techniques related to crossed products by actions of locally compact quantum groups. In particular we show existence of a ``braided comultiplication'' on the algebra spanned by slices of ``braided multiplicative unitary''.
\end{abstract}

\maketitle

\section{Introduction}\label{intro}

This paper is devoted to the study of the notion of a semidirect product of locally compact quantum groups. Before introducing all relevant definitions let us briefly consider the classical case: let $G$ be a locally compact group. It is elementary that $G$ is (isomorphic to) a semidirect product of certain two of its closed subgroups $K$ and $H$ if and only if there exists a continuous homomorphism $\rh\colon{G}\to{G}$ such that $\rh\comp\rh=\rh$. Indeed, given $\rh$ we define $K=\ker{\rh}$ and $H=\mathrm{im}\,\rh$ and then any element $g\in{G}$ can be expressed as $kh$, where $h=\rh(g)\in{H}$ and $k=gh^{-1}\in{K}$. Since $K$ is a normal subgroup of $G$, we conclude that $G\cong{K}\rtimes{H}$. Equivalently one can say that $G$ has the structure of a semidirect product if and only if there exists a locally compact group $H$ and a pair of continuous homomorphisms
\[
\rh\colon{G}\longrightarrow{H}\quad\text{and}\quad\imath\colon{H}\longrightarrow{G}
\]
such that
\begin{equation}\label{rhoiota0}
\rh\comp\imath=\id_{H}.
\end{equation}

This last point of view on semidirect products will serve as the basis of our approach to semidirect products of locally compact quantum groups. We will introduce the concept of a locally compact quantum group with projection onto another locally compact quantum group (Definition \ref{GprojH}). This will generalize existence of maps $\rh$ and $\imath$ as above. In particular, we will show that if $\GG$ is a locally compact quantum group with projection onto a locally compact quantum group $\HH$ then $\HH$ is a closed quantum subgroup of $\GG$ and also a ``quotient'' of $\GG$. This corresponds exactly to the fact that maps $\rh$ and $\imath$ satisfying \eqref{rhoiota0} are surjective and injective respectively.

The approach of this paper to semidirect products of (locally compact) quantum groups is inspired by the thesis of S.~Roy (\cite{RoyPhd}). However, our approach is based on a different axiomatization of the same concept. Moreover while in \cite{RoyPhd} the quantum groups under investigation are objects defined by modular multiplicative unitaries and are studied exclusively on the level of \cst-algebras, we work with locally compact quantum groups and prefer the von Neumann algebraic framework. In particular we will often use the ultra-weak topology of a von Neumann algebra. The symbol $\linW{X}$ will denote the ultra-weak closure of the span of a subset $X$ of a von Neumann algebra. Similarly $\overline{X}^{\:\!\text{\tiny\rm{w}}}\:\!\!$ will denote the ultra-weak closure of $X$. The von Neumann algebra tensor product will be denoted by the symbol $\vtens$, while tensor product of vector spaces, Hilbert spaces and minimal tensor products of \cst-algebras will be written as $\tens$. The flip map on tensor products (of any of the above mentioned objects) will be denoted by $\flip$. The \emph{leg numbering notation}, by now a standard in quantum group theory, will be used throughout the paper.

Let us briefly describe the contents of the paper: Section \ref{prelim} contains basic information and notation pertaining to locally compact quantum groups (Subsection \ref{lcqgs}), the opposite and commutant locally compact quantum groups (Subsection \ref{oppcomm}), crossed products of von Neumann algebras by locally compact quantum group actions (Subsection \ref{cp}). Aside from their applications in later sections, a number of results from Subsection \ref{cp} are of independent interest. Particularly Proposition \ref{stronPodles} which shows thats an appropriate form of Podle\'s condition for actions of quantum groups on \cst-algebras (\cite{Sym}) has a von Neumann algebraic analog which, moreover, is always an automatic consequence of a von Neumann algebraic definition of an action.

Section \ref{QGP} introduces locally compact quantum groups with projection and basic results about these structures. Afterwards, in Section \ref{emer} we show among other things that if $\GG$ is a locally compact quantum group with projection onto $\HH$ then the von Neumann algebra $\Linf(\GG)$ is a crossed product of a certain subalgebra $\sN\subset\Linf(\GG)$ by an action of $\hh{\HH}^\op$. Moreover we identify $\sN$ as the closure of the set of slices of the ``braided multiplicative unitary'' (see \cite[Section 6]{RoyPhd}) and show that $\Delta_\GG$ restricts to a map $\sN\to\sN\boxtimes\sN$, where the latter algebra is a certain ``twist'' of the usual tensor product $\sN\vtens\sN$. The results are illustrated by examples in which both classical and quantum semidirect products are described. For more examples and a much more thorough treatment on the \cst-algebra level we refer to \cite[Sections 6.2.2--6.2.4]{RoyPhd}.

\section*{Acknowledgements}

The authors would like to express deep gratitude to Sutanu Roy, whose work inspired this paper, for fruitful discussions and suggestions. We would also like to thank Ralf Meyer and Stanis\l{}aw Woronowicz for stimulating exchanges on the subject of semidirect products of quantum groups, braided quantum groups and multiplicative unitaries. In particular the proof of Theorem \ref{glowne} was completed along the lines indicated to us by S.~Roy. In addition we would like to thank Adam Skalski for encouragement to pursue our investigation.

\section{Preliminaries}\label{prelim}

\subsection{Locally compact quantum groups}\label{lcqgs}

Let $\GG$ be a locally compact quantum group as defined by J.~Kustermans and S.~Vaes in \cite[Definition 1.1]{KVvN} (cf.~also \cite{KV,mnw}). The von Neumann algebra corresponding to $\GG$ will be denoted by the symbol $\Linf(\GG)$ and the comultiplication $\Linf(\GG)\to\Linf(\GG)\vtens\Linf(\GG)$will be $\Delta_\GG$. The Hilbert space defined via the GNS construction from the right Haar weight of $\GG$ will be denoted by $\Ltwo(\GG)$. As usual $\hh{\GG}$ will denote the dual locally compact quantum group. The Hilbert spaces $\Ltwo(\hh{\GG})$ and $\Ltwo(\GG)$ can (and will) be identified, so in particular, we have $\Linf(\GG),\Linf(\hh{\GG})\subset\B\bigl(\Ltwo(\GG)\bigr)$. The modular conjugations related to right Haar measures of $\GG$ and $\hh{\GG}$ will be denoted by $J$ and $\hh{J}$ respectively. The antiunitary operator $\hh{J}$ enters the formula for the \emph{unitary antipode} $R$ of $\GG$. More precisely we have
\[
\begin{aligned}
R(x)&=\hh{J}x^*\hh{J},&\quad{x}\in\Linf(\GG).
\end{aligned}
\]

The \emph{Kac-Takesaki operator} (i.e.~the right regular representation) of $\GG$ will be denoted by $W^\GG$. It is an element of $\Linf(\hh{\GG})\vtens\Linf(\GG)$ and satisfies
\[
\begin{split}
(\id\tens\Delta_\GG)W^\GG&=W^\GG_{12}{W_{13}^\GG},\\
(\Delta_{\hh{\GG}}\tens\id)W^\GG&=W_{23}^\GG{W_{13}^\GG}.
\end{split}
\]

\subsection{The opposite and commutant quantum group}\label{oppcomm}

In this subsection we will recall some material covered in \cite[Section 4]{KVvN}. Given a locally compact quantum group $\GG$ one can define two more associated quantum groups. The first one is the \emph{opposite quantum group} $\GG^\op$ defined by $\Linf(\GG^\op)=\Linf(\GG)$ and $\Delta_{\GG^\op}=\flip\comp\Delta_\GG$. The second is the \emph{commutant quantum group} $\GG'$ defined by $\Linf(\GG')=\Linf(\GG)'$ and
\[
\begin{aligned}
\Delta_{\GG'}(x')&=(J\tens{J})\Delta_\GG(Jx'J)(J\tens{J}),&\quad{x'}\in\Linf(\GG').
\end{aligned}
\]
It is easy to see that the Hilbert spaces corresponding to $\GG^\op$ and $\GG'$ can be naturally identified.

Similarly one can define $\hh{\GG}^\op$ and $\hh{\GG}'$. It turns out that
\[
\begin{aligned}
\hh{\GG^\op}&=\hh{\GG}',&\quad\hh{\GG'}&=\hh{\GG}^\op,&\quad(\GG')^\op&=(\GG^\op)'
\end{aligned}
\]
(see \cite{KVvN} for details).

Let $\Phi$ be the map $\B\bigl(\Ltwo(\GG)\bigr)\to\B\bigl(\Ltwo(\GG)\bigr)$ defined as $\Phi(y)=J\hh{J}y\hh{J}J$. When appropriately restricted, $\Phi$ provides an isomorphism of quantum groups between $\GG^\op$ and $\GG'$ and between $\hh{\GG}^\op$ and $\hh{\GG}'$ (\cite{KVvN,KaspSol}).

\subsection{Crossed products}\label{cp}

Let $\GG$ be a locally compact quantum group. Let $\sN$ be a von Neumann algebra and let $\act\colon\sN\to\sN\vtens\Linf(\GG)$ be an action of $\GG$ on $\sN$, i.e.~a normal unital $*$-homomorphism such that $(\act\tens\id)\comp\act=(\id\tens\Delta_\GG)\comp\act$. The \emph{crossed product} of $\sN$ by (the action $\act$ of) $\GG$ is the von Neumann subalgebra of $\sN\vtens\B\bigl(\Ltwo(\GG)\bigr)$ generated by $\act(\sN)$ and $\I\tens\Linf(\hh{\GG})$. It is denoted by the symbol $\sN\rtimes_\act\GG$. 

The crossed product admits a distinguished action of $\hh{\GG}^\op$ called the \emph{dual action}. This action, denoted by $\hh{\act}\colon\sN\rtimes_\act\GG\to(\sN\rtimes_\act\GG)\vtens\Linf(\hh{\GG}^\op)$ is defined by
\[
\begin{aligned}
\hh{\act}\bigl(\act(x)\bigr)&=\act(x)\tens\I,&\quad{x}\in\Linf(\GG),\\
\hh{\act}(\I\tens{y})&=\I\tens\Delta_{\hh{\GG}^\op}(y),&\quad{y}\in\Linf(\hh{\GG}).
\end{aligned}
\]

To simplify notation let us temporarily denote the crossed product $\sN\rtimes_\act\GG$ by the symbol $\sM$. Thus we have an action $\hh{\act}$ of $\hh{\GG}^\op$ on $\sM$. Moreover if we put $X=\I\tens{W^\GG}\in\sM\vtens\Linf(\GG)$ then $X$ satisfies
\begin{align}
(\id\tens\Delta_\GG)X&=X_{12}X_{13},\nonumber\\
(\hh{\act}\tens\id)X&=X_{13}W_{23}^\GG.\label{XX2}
\end{align}
To see \eqref{XX2} we compute:
\[
(\hh{\act}\tens\id)X=(\id\tens\Delta_{\hh{\GG}^\op}\tens\id)W_{23}^\GG=W_{24}^\GG{W_{34}^\GG}=X_{13}W_{23}^\GG.
\]

The structure $(\sM,\hh{\act},X)$ determines $\sN$ and $\act$. More precisely we have the following consequence of (a special case of) \cite[Proposition 1.22]{VaesVainerman}:

\begin{proposition}\label{qLan}
Let $\sM$ be a von Neumann algebra and let $\Act\colon\sM\to\sM\vtens\Linf(\hh{\GG}^\op)$ be an action of $\hh{\GG}^\op$ on $\sM$. Let $X\in\sM\vtens\Linf(\GG)$ be a unitary element such that
\[
\begin{split}
(\id\tens\Delta_\GG)X&=X_{12}X_{13},\\
(\Act\tens\id)X&=X_{13}W_{23}^\GG.
\end{split}
\]
Denote by $\sN$ the fixed point subalgebra of $\sM$:
\[
\sN=\bigl\{x\in\sM\st\Act(x)=x\tens\I\bigr\}.
\]
Then the map $\act\colon\sN\to\sN\vtens\Linf(\GG)$ defined by
\[
\act(x)=X(x\tens\I)X^* 
\]
is an action of $\GG$ on $\sN$ and there exists an isomorphism $\Psi\colon\sM\to\sN\rtimes_\act\GG$ such that
\begin{itemize}
\item[$\blacktriangleright$] $\hh{\act}\comp\Psi=(\Psi\tens\id)\comp\Act$,
\item[$\blacktriangleright$] for any $y\in\sN$ we have $\Psi(y)=\act(y)$,
\item[$\blacktriangleright$] $(\Psi\tens\id)X=W_{23}^\GG$.
\end{itemize}
\end{proposition}

Proposition \ref{qLan} is a quantum group version of Landstad's theorem for classical covariant systems (\cite{lan}). We will use a modified version of this proposition in which the commutant and opposite quantum groups of a locally compact quantum group $\hh{\HH}$ will be identified via the isomorphism $\Phi$ of Subsection \ref{oppcomm}.

\begin{corollary}\label{cpProp}
Let $\sM$ be a von Neumann algebra and let $\Act\colon\sM\to\sM\vtens\Linf(\HH)$ be an action of a locally compact quantum group $\HH$ on $\sM$. Let $Y\in\Linf(\hh{\HH})\vtens\sM$ be a unitary element such that
\[
\begin{split}
(\Delta_{\hh{\HH}}\tens\id)Y&=Y_{23}Y_{13},\\
(\id\tens\Act)Y&=Y_{12}W_{13}^\HH
\end{split} 
\]
and denote by $\sN$ the fixed point subalgebra of $\sM$:
\[
\sN=\bigl\{x\in\sM\st\Act(x)=x\tens\I\bigr\}.
\]
Define $\act\colon\sN\to\sN\vtens\Linf(\hh{\HH}^\op)$ by
\[
\begin{aligned}
\act(y)&=\flip\bigl(Y(\I\tens{y})Y^*\bigr),&\quad{y}\in\sN.
\end{aligned}
\]
Then $\act$ is an action of $\hh{\HH}^\op$ and $\sM$ is isomorphic to $\sN\rtimes_\act\hh{\HH}^\op$ in such a way that $\Act$ is identified with the dual action of
\[
{\hh{\hh{\HH}^{\raisebox{.27ex}{\op}}}}^\op=\Bigl(\hh{\hh{\HH}}{\,}'\Bigr)^\op=(\HH')^\op\cong(\HH^\op)^\op=\HH.
\]

The isomorphism $\Psi\colon\sM\to\sN\rtimes_\act\hh{\HH}^\op$ satisfies
\begin{itemize}
\item[$\blacktriangleright$] $\hh{\act}\comp\Psi=(\Psi\tens\Phi)\comp\Act$,
\item[$\blacktriangleright$] $\Psi(y)=\act(y)$ for all $y\in\sN$,
\item[$\blacktriangleright$] $(\Phi\tens\Psi)Y=W_{13}^\HH$.
\end{itemize}
\end{corollary}

The next proposition giving a simplified picture of crossed products by quantum group actions seems to be well known, but we were unable to find it stated explicitly in the literature (see Remark \ref{after}). Using it we will prove a result which will be of importance in Section \ref{emer}.

\begin{proposition}\label{span}
Let $\GG$ be a locally compact quantum group acting on a von Neumann algebra $\sN$ via $\act\colon\sN\to\sN\vtens\Linf(\GG)$. Then the crossed product $\sN\rtimes_\act\GG$ is the ultra-weakly closed linear span of elements of the form $\act(y)(\I\tens{x})$ with $y\in\sN$ and $x\in\Linf(\hh{\GG})$:
\begin{equation}\label{NrtimesG}
\sN\rtimes_\act\GG=\linW\bigl\{\act(y)(\I\tens{x})\st{y}\in\sN,\:x\in\Linf(\hh{\GG})\bigr\}.
\end{equation}
\end{proposition}

\begin{proof}
Denote the right hand side of \eqref{NrtimesG} by $\sX$ and let
\[
\cA_0=\lin\bigl\{\act(y)\bigl(\I\tens[(\id\tens\omega_{\xi,\eta})W^\GG]\bigr)\st{y}\in\sN,\:\xi,\eta\in\Ltwo(\GG)\bigr\}.
\]
Note that $\cA_0$ is ultra-weakly dense in $\sX$.

Now take $y\in\sN$ and $\xi,\eta\in\Ltwo(\GG)$. Let us compute the following product
\[
\begin{split}
\bigl(\I\tens[(\id\tens\omega_{\xi,\eta})W^\GG]\bigr)\act(y)
&=(\id\tens\id\tens\omega_{\xi,\eta})(W_{23}^\GG\act(y)_{12})\\
&=(\id\tens\id\tens\omega_{\xi,\eta})(W_{23}^\GG\act(y)_{12}{W_{23}^\GG}^*W_{23}^\GG)\\
&=(\id\tens\id\tens\omega_{\xi,\eta})\bigl((\id\tens\Delta_\GG)\act(y)\cdot{W_{23}^\GG}\bigr)\\
&=(\id\tens\id\tens\omega_{\xi,\eta})\bigl((\act\tens\id)\act(y)\cdot{W_{23}^\GG}\bigr).
\end{split}
\]
Let $\{e_i\}$ be an orthonormal basis of $\Ltwo(\GG)$. Inserting $\I=\sum\limits_i\ket{e_i}\!\bra{e_i}$ after $\act(y)$ on the right hand side of the above expression we obtain
\[
\bigl(\I\tens[(\id\tens\omega_{\xi,\eta})W^\GG]\bigr)\act(y)=
\sum_i\act\bigl((\id\tens\omega_{\xi,e_i})\act(y)\bigr)\bigl(\I\tens[(\id\tens\omega_{e_i,\eta})W^\GG]\bigr).
\]
It follows from this that the $*$-algebra generated by $\cA_0$ is contained in $\sX$. Let $\cA$ be the weak closure of $\cA_0$. Then $\cA$ is a von Neumann algebra contained in $\sX$. Moreover $\cA$ contains $\act(\sN)$ and $\I\tens\Linf(\hh{\GG})$. Thus $\cA\supset\sN\rtimes_\act\GG$. On the other hand $\sX\subset\sN\rtimes_\act\GG$, so we must have $\cA=\sX=\sN\rtimes_\act\GG$.
\end{proof}

\begin{remark}\label{after}
Results very similar to Proposition \ref{span} are present in the literature in the \cst-algebraic context (e.g.~\cite[Lemme 7.2]{BS}, or the computation in \cite[Page 191]{SkZa}). However, we cannot use them in any way to prove Proposition \ref{span} because the \cst-algebraic notion of an action $\act$ of a locally compact quantum group $\GG$ on a \cst-algebra $\sA$ demands that elements of the form $\act(y)(\I\tens{x})$ with $y\in\sA$ and $x\in\C_0(\GG)$ lie in $\sA\tens\C_0(\GG)$. This is a condition on $\act$ which we do not assume.
\end{remark}

Let $\act$ be an action of a locally compact quantum group $\GG$ on a von Neumann algebra $\sN$. An appropriate modification of \cite[Proposition 1.20]{VaesVainerman} says that $(\sN\rtimes_\act\GG)\rtimes_{\hh{\act}}\hh{\GG}^\op$ is isomorphic to $\sN\vtens\B\bigl(\Ltwo(\GG)\bigr)$ is such a way that
\begin{itemize}
\item[$\blacktriangleright$] for $u\in\sN\rtimes_\act\GG$ the element $\hh{\act}(u)$ of $(\sN\rtimes_\act\GG)\rtimes_{\hh{\act}}\hh{\GG}^\op$ is mapped to $u\in\sN\vtens\B\bigl(\Ltwo(\GG)\bigr)$,
\item[$\blacktriangleright$] for any $v\in\Linf(\hh{\hh{\GG}^\op})=\Linf(\GG)'$ the element $(\I\tens{v})$ of $(\sN\rtimes_\act\GG)\rtimes_{\hh{\act}}\hh{\GG}^\op$ is mapped to $(\I\tens{v})\in\I\tens\Linf(\GG)'\subset\sN\vtens\B\bigl(\Ltwo(\GG)\bigr)$.
\end{itemize}
We may apply Proposition \ref{span} to the crossed product of $\sN\rtimes_\act\GG$ by $\hh{\GG}^\op$ to conclude that
\begin{equation}\label{NBH}
\sN\vtens\B\bigl(\Ltwo(\GG)\bigr)=\linW\bigl\{u(\I\tens{v})\st{u}\in\sN\rtimes_\act\GG,\:v\in\Linf(\GG)'\bigr\}.
\end{equation}

Taking slices of the right legs of both sides of \eqref{NBH} with normal functionals we immediately find that
\begin{equation}\label{N2}
\sN\subset\linW\bigl\{(\id\tens\omega)(u)\st{u}\in\sN\rtimes_\act\GG,\:\omega\in\B(\Ltwo(\GG))_*\bigr\}.
\end{equation}
On the other hand, the right hand side of \eqref{N2} is contained in $\sN$, because $\sN\rtimes_\act\GG\subset\sN\vtens\B\bigl(\Ltwo(\GG)\bigr)$. We can summarize these considerations in the next proposition.

\begin{proposition}\label{Nslice}
Let $\GG$ be a locally compact quantum group acting on a von Neumann algebra $\sN$ via $\act\colon\sN\to\sN\vtens\Linf(\GG)$. Then
\[
\linW\bigl\{(\id\tens\omega)(u)\st{u}\in\sN\rtimes_\act\GG,\:\omega\in\B(\Ltwo(\GG))_*\bigr\}=\sN.
\]
\end{proposition}

The next proposition and Corollary \ref{invN} will be useful tools in section \ref{emer}.

\begin{proposition}\label{dowcN}
Let $\GG$ be a locally compact quantum group acting on a von Neumann algebra $\sN$ via $\act\colon\sN\to\sN\vtens\Linf(\GG)$ and let $\cN$ be a subset of $\sN$ such that
\begin{equation}\label{N3}
\linW\bigl\{\act(y)(\I\tens{x})\st{y}\in\cN,\:x\in\Linf(\hh{\GG})\bigr\}=\sN\rtimes_\act\GG.
\end{equation}
Then
\[
\sN=\linW\bigl\{(\id\tens\omega)\act(y)\st{y}\in\cN,\:\omega\in\B(\Ltwo(\GG))_*\bigr\}.
\]
\end{proposition}

\begin{proof}
Applying all $\omega\in\B\bigl(\Ltwo(\GG)\bigr)_*$ to the right legs of both sides of \eqref{N3} and using Proposition \ref{Nslice} we see that
\begin{equation}\label{N1}
\sN\subset\linW\bigl\{(\id\tens\omega)\act(y)\st{y}\in\cN,\:\omega\in\B(\Ltwo(\GG))_*\bigr\}
\end{equation}
But for each $y\in\cN$ the element $\act(y)$ belongs to $\sN\vtens\Linf(\GG)$, so the right hand side of \eqref{N1} is contained in $\sN$.
\end{proof}

By applying Proposition \ref{N3} to the special case $\cN=\sN$ we obtain the next result which says that all von Neumann algebraic actions of locally compact quantum groups on von Neumann algebras satisfy an analog of a weakened form of Podle\'s condition. 

\begin{corollary}\label{weakPodles}
Let $\GG$ be a locally compact quantum group acting on a von Neumann algebra $\sN$ via $\act\colon\sN\to\sN\vtens\Linf(\GG)$. Then
\[
\sN=\linW\bigl\{(\id\tens\omega)\act(y)\st{y}\in\sN,\:\omega\in\B(\Ltwo(\GG))_*\bigr\}.
\]
\end{corollary}

If one additionally assumes in Proposition \ref{dowcN} an appropriate form of ``invariance'' of the subset $\cN$ we get a stronger result:

\begin{corollary}
In the situation from Proposition \ref{dowcN}, if $\cN$ is a subspace of $\sN$ such that
\begin{equation}\label{invN}
\act(\cN)\subset\overline{\cN\tens\Linf(\GG)}^{\:\!\text{\tiny\rm{w}}}
\end{equation}
then $\cN$ is ultra-weakly dense in $\sN$.
\end{corollary}

\begin{proof}
We know that
\[
\sN=\linW\bigl\{(\id\tens\omega)\act(y)\st{y}\in\cN,\:\omega\in\B(\Ltwo(\GG))_*\bigr\}.
\]
However, by \eqref{invN} the right hand side is contained in $\overline{\cN}^{\:\!\text{\tiny\rm{w}}}\!\!$, which is contained in $\sN$.
\end{proof}

Finally, as an application of Corollary \ref{weakPodles} let us argue that for actions of locally compact quantum groups on von Neumann algebra level an appropriate form of Podle\'s condition is automatically satisfied:

\begin{proposition}\label{stronPodles}
Let $\GG$ be a locally compact quantum group acting on a von Neumann algebra $\sN$ via $\act\colon\sN\to\sN\vtens\Linf(\GG)$. Then
\[
\linW\bigl\{\act(y)(\I\tens{x})\st{x}\in\Linf(\GG),\:y\in\sN\bigr\}=\sN\vtens\Linf(\GG).
\]
\end{proposition}

\begin{proof}
First we recall that $\Linf(\GG)\Linf(\hh{\GG})$ is ultra-weakly dense in $\B\bigl(\Ltwo(\GG)\bigr)$ (\cite[Proposition 2.5]{VV}). This can be used instead of \cite[Proposition 3.2 \emph{b})]{BS} for the proof of the von Neumann algebraic version of \cite[Proposition 3.6 \emph{b})]{BS}:
\begin{equation}\label{regu}
\linW\bigl\{(b\tens\I)W^\GG(\I\tens{x})\st{x}\in\Linf(\GG),\:b\in\B(\Ltwo(\GG))\bigr\}=\B\bigl(\Ltwo(\GG))\vtens\Linf(\GG).
\end{equation}
By taking adjoints all elements of \eqref{regu}, applying the flip $\flip$, conjugating with $\hh{J}\tens\hh{J}$ and taking adjoints again we obtain
\begin{equation}\label{regu2}
\linW\bigl\{(\I\tens{b})V^*(x\tens\I)\st{x}\in\Linf(\GG),\:b\in\B(\Ltwo(\GG))\bigr\}=\Linf(\GG)\vtens\B\bigl(\Ltwo(\GG)),
\end{equation}
where $V=(\hh{J}\tens\hh{J})\flip({W^\GG}^*)(\hh{J}\tens\hh{J})\in\Linf(\GG)\vtens\Linf(\hh{\GG})'$. It is important to note that $V$ implements the comultiplication $\Delta_\GG$ via
\[
\begin{aligned}
\Delta(x)&=V^*(\I\tens{x})V,&\quad{x}\in\Linf(\GG).
\end{aligned}
\]

Now it is enough to repeat the proof of the first part of \cite[Proposition 5.8]{BSV} in our context: using in the first step Corollary \ref{weakPodles}, \eqref{regu2} in the ninth and Corollary \ref{weakPodles} again in the tenth, we obtain
\[
\begin{split}
\linW&\bigl\{\alpha(y)(\I\tens{x})\st{x}\in\Linf(\GG),\:y\in\sN\bigr\}\\
&=\linW\bigl\{\alpha((\id\tens\omega)\alpha(y))(\I\tens{x})\st{x}\in\Linf(\GG),\:y\in\sN,\:\omega\in\B(\Ltwo(\GG))_*\bigr\}\\
&=\linW\bigl\{(\id\tens\id\tens\omega)((\id\tens\Delta_\GG)\alpha(y))(\I\tens{x})\st{x}\in\Linf(\GG),\:y\in\sN,\:\omega\in\B(\Ltwo(\GG))_*\bigr\}\\
&=\linW\bigl\{(\id\tens\id\tens\omega)(V^*_{23}\alpha(y)_{13}V_{23}(\I\tens{x}\tens\I))\st{x}\in\Linf(\GG),\:y\in\sN,\:\omega\in\B(\Ltwo(\GG))_*\bigr\}\\
&=\linW\bigl\{(\id\tens\id\tens\omega)(V^*_{23}\alpha(y)_{13}V_{23}(\I\tens{x}\tens{b}))\st\\
&\qquad\qquad\qquad\qquad\qquad\qquad\qquad\qquad{x}\in\Linf(\GG),\:y\in\sN,\:\omega\in\B(\Ltwo(\GG))_*,\:b\in\B(\Ltwo(\GG))\bigr\}\\
&=\linW\bigl\{(\id\tens\id\tens\omega)(V^*_{23}\alpha(y)_{13}(\I\tens{x}\tens{b}))\st\\
&\qquad\qquad\qquad\qquad\qquad\qquad\qquad\qquad{x}\in\Linf(\GG),\:y\in\sN,\:\omega\in\B(\Ltwo(\GG))_*,\:b\in\B(\Ltwo(\GG))\bigr\}\\
&=\linW\bigl\{(\id\tens\id\tens\omega)(V^*_{23}\alpha(y)_{13}(\I\tens{x}\tens\I))\st{x}\in\Linf(\GG),\:y\in\sN,\:\omega\in\B(\Ltwo(\GG))_*\bigr\}\\
&=\linW\bigl\{(\id\tens\id\tens\omega)(V^*_{23}(\I\tens{x}\tens\I)\alpha(y)_{13})\st{x}\in\Linf(\GG),\:y\in\sN,\:\omega\in\B(\Ltwo(\GG))_*\bigr\}\\
&=\linW\bigl\{(\id\tens\id\tens\omega)((\I\tens\I\tens{b})V^*_{23}(\I\tens{x}\tens\I)\alpha(y)_{13})\st\\
&\qquad\qquad\qquad\qquad\qquad\qquad\qquad\qquad{x}\in\Linf(\GG),\:y\in\sN,\:\omega\in\B(\Ltwo(\GG))_*,\:b\in\B(\Ltwo(\GG))\bigr\}\\
&=\linW\bigl\{(\id\tens\id\tens\omega)((\I\tens{x}\tens{b})\alpha(y)_{13})\st\\
&\qquad\qquad\qquad\qquad\qquad\qquad\qquad\qquad{x}\in\Linf(\GG),\:y\in\sN,\:\omega\in\B(\Ltwo(\GG))_*,\:b\in\B(\Ltwo(\GG))\bigr\}\\
&=\linW\bigl\{(\I\tens{x})(y\tens\I)\st{x}\in\Linf(\GG),\:y\in\sN\bigr\}=\sN\vtens\Linf(\GG).
\end{split}
\]
\end{proof}

\subsection{Homomorphisms of quantum groups}\label{Hom}

We will follow \cite{MRW} and say that given two locally compact quantum groups $\HH$ and $\GG$, a \emph{homomorphism} from $\HH$ to $\GG$ is described by a \emph{bicharacter} from $\HH$ to $\GG$, i.e.~a unitary $V\in\M\bigl(\C_0(\hh{\GG})\tens\C_0(\HH)\bigr)$ such that
\[
(\Delta_{\hh{\GG}}\tens\id)V=V_{23}V_{13}\quad\text{and}\quad(\id\tens\Delta_\HH)V=V_{12}V_{13}.
\]
There is a bijection between bicharacters from $\HH$ to $\GG$ and \emph{right quantum homomorphisms} from $\HH$ to $\GG$ which are elements $\rho\in\Mor\bigl(\C_0(\GG),\C_0(\GG)\tens\C_0(\HH)\bigr)$ such that
\begin{equation}\label{rqh}
(\Delta_\GG\tens\id)\comp\rho=(\id\tens\rho)\comp\Delta_\GG\quad\text{and}\quad(\id\tens\Delta_\HH)\comp\rho=(\rho\tens\id)\comp\rho.
\end{equation}
Given a bicharacter $V$ as above the corresponding right quantum homomorphism $\rho$ is defined by
\[
\begin{aligned}
\rho(a)&=V(a\tens\I)V^*,&\quad{a}\in\C_0(\GG).
\end{aligned}
\]
Clearly $\rho$ extends to a normal unital $*$-homomorphism $\Linf(\GG)\to\Linf(\GG)\vtens\Linf(\HH)$ and formulas \eqref{rqh} still hold with all maps appropriately extended. We will call such a map $\rho$ a \emph{\wst-right quantum homomorphism} from $\HH$ to $\GG$.

Repeating the steps of \cite[Proof of theorem 5.3]{MRW} we obtain the following proposition:

\begin{proposition}\label{wstrqh}
Let $\rho\colon\Linf(\GG)\to\Linf(\GG)\vtens\Linf(\HH)$  be a \wst-right quantum homomorphism from $\HH$ to $\GG$. Then there exists a unique bicharacter $V$ from $\HH$ to $\GG$ such that
\[
\begin{aligned}
\rho(x)&=V(x\tens\I)V^*,&\quad{x}\in\Linf(\GG).
\end{aligned}
\]
\end{proposition}

It follows from Proposition \ref{wstrqh} that \wst-right quantum homomorphisms from $\HH$ to $\GG$ are in bijection with right quantum homomorphisms from $\HH$ to $\GG$ and, consequently, also with bicharacters from $\HH$ to $\GG$.

\section{Quantum groups with projection}\label{QGP}

\begin{definition}\label{GprojH}
Let $\GG$ and $\HH$ be locally compact quantum groups. We say that $\GG$ is a \emph{quantum group with projection onto $\HH$} if there exists a unital normal $*$-homomorphism $\pi\colon\Linf(\HH)\to\Linf(\GG)$ and a \wst-right quantum homomorphism $\th\colon\Linf(\GG)\to\Linf(\GG)\vtens\Linf(\HH)$ such that
\begin{equation}\label{pitenspi}
(\pi\tens\pi)\comp\Delta_\HH=\Delta_\GG\comp\pi
\end{equation}
and
\begin{equation}\label{thpi}
\th\comp\pi=(\pi\tens\id)\comp\Delta_\HH.
\end{equation}
\end{definition}

\begin{example}\label{ClassEx1}
Let $G$ and $H$ be locally compact groups and let $\rh\colon{G}\to{H}$ and $\imath\colon{H}\to{G}$ be continuous homomorphisms such that
\begin{equation}\label{rhoiota}
\rh\comp\imath=\id_{H}.
\end{equation}
As mentioned in the introduction, this means that $G$ is isomorphic to a semidirect product of $H$ and $K=\ker{\rh}$. Let $\pi$ be the unital normal $*$-homomorphism $\Linf(H)\to\Linf(G)$ corresponding to the homomorphism $\rh$:
\[
\begin{aligned}
\pi(f)&=f\comp\rh,&\quad{f}\in\Linf(H)
\end{aligned}
\]
and let $\th\colon\Linf(G)\to\Linf(G)\vtens\Linf(H)\cong\Linf(G\times{H})$ be the \wst-right quantum homomorphism associated to the homomorphism $\imath$:
\[
\begin{aligned}
\th(f)(g,h)&=f\bigl(g\cdot\imath(h)\bigr),&\quad{f}\in\Linf(G),\;(g,h)\in{G}\times{H}.
\end{aligned}
\]
It is easy to verify that \eqref{pitenspi} and \eqref{thpi} are satisfied. 

Clearly, it follows from \eqref{rhoiota} $\rh$ is surjective and $\imath$ is injective. Thus $\imath(H)$ is a closed subgroup of $G$ and $H$ itself is a quotient of $G$. The latter fact is reflected in injectivity of $\pi$.
\end{example}

As explained above, in the classical situation described in Example \ref{ClassEx1}, the maps  $\rh$ and $\imath$ satisfying \eqref{rhoiota} are surjective and injective respectively. The same phenomenon can be seen on the level of quantum groups. This is the content of Proposition \ref{injpi} and Theorem \ref{subgroup} below.

\begin{proposition}\label{injpi}
Let $\GG$ and $\HH$ be locally compact quantum groups. Assume further that $\GG$ is a quantum group with projection onto $\HH$ with corresponding maps $\pi$ and $\th$. Then $\pi$ is injective.
\end{proposition}

\begin{proof}
Let  $p\in\Linf(\HH)$ be a central projection such that $\ker\pi=p\Linf(\HH)$. For any $x\in\ker{\pi}$ we have $(\pi\tens\id)\Delta_\HH(x)=\th\bigl(\pi(x)\bigr)=0$ which proves that $\Delta_\HH(\ker{\pi})\subset\ker{\pi}\vtens\Linf(\HH)$. In particular
\[
(p\tens\I)\Delta_\HH(p)=\Delta_\HH(p),
\]
so that
\begin{equation}\label{ppp}
\Delta_\HH(p)\leq{p}\tens\I
\end{equation}
Let $W^\HH$ be the Kac-Takesaki operator for $\HH$ and let $J^\HH$ and $\hh{J}^\HH$ be the modular conjugations for the right Haar weights of $\HH$ and $\hh{\HH}$ respectively. Conjugating both sides of \eqref{ppp} by $(J^\HH\tens\hh{J}^\HH)$ and using centrality of $p$ we obtain
\[
(J^\HH\tens\hh{J}^\HH)W^\HH(p\tens\I){W^\HH}^*(J^\HH\tens\hh{J}^\HH)=(J^\HH\tens\hh{J}^\HH)\Delta_\HH(p)(J^\HH\tens\hh{J}^\HH)\leq{p}\tens\I
\]
(cf.~\cite[Section 10.13, Corollary 1]{SZ}).

Moreover, since $(J^\HH\tens\hh{J}^\HH)W^\HH(J^\HH\tens\hh{J}^\HH)={W^\HH}^*$, we have
\[
{W^\HH}^*(p\tens\I)W^\HH\leq{p}\tens\I.
\]
Conjugating both sides of this inequality by $W^\HH$ we obtain
\[
p\tens\I\leq{W^\HH}(p\tens\I){W^\HH}^*=\Delta_\HH(p).
\]
Thus $\Delta_\HH(p)=p\tens\I$. Since the action of $\HH$ on itself is ergodic, we conclude that $p\in\CC\I$ (cf.~\cite[Theorem 2.1]{MRW}). The case $p=\I$ contradicts unitality of $\pi$, so $p$ must be the zero projection and consequently $\ker\pi=\{0\}$.
\end{proof}

We continue investigating the situation when $\GG$ is a locally compact quantum group with a projection onto $\HH$ with associated maps $\pi$ and $\th$. Since $\pi$ is injective, we can identify $\Linf(\HH)$ with its image inside $\Linf(\GG)$. Now let $U\in\M\bigl(\C_0(\hh{\GG})\tens\C_0(\HH)\bigr)$ be the bicharacter corresponding to the \wst-right quantum homomorphism $\th$. As $U$ implements $\th$ (cf.~Subsection \ref{Hom}) we have
\[
(\id\tens\th)U=U_{23}U_{12}U_{23}^*.
\]
On the other hand, once we regard $\pi$ as identity, formula \eqref{thpi} implies that $\th=\Delta_\HH$, so that
\begin{equation}\label{dwa}
(\id\tens\th)U=(\id\tens\Delta_\HH)U=U_{12}U_{13}.
\end{equation}
It follows that $U_{23}U_{12}U_{23}^*=U_{12}U_{13}$, so that the unitary $U\in\B\bigl(\Ltwo(\GG)\tens\Ltwo(\GG)\bigr)$ is multiplicative. In particular it follows that the scheme introduced in Definition \ref{GprojH} gives rise to an example of S.~Roy's \cst-quantum groups with projection (\cite[Definition 3.35]{RoyPhd}). Therefore, by \cite[Proposition 3.36]{RoyPhd} the operator $U\in\B\bigl(\Ltwo(\GG)\tens\Ltwo(\GG)\bigr)$ is a manageable multiplicative unitary.

It is known that each manageable multiplicative unitary defines a quantum group which, in principle, could fail to be a \emph{locally compact} quantum group as defined in \cite[Definition 4.1]{KV} (or equivalently \cite[Definition 1.1]{KVvN} or \cite[Definition 1.5]{mnw}). However, the theory of such quantum groups is also very rich (cf.~\cite{mmu,SolWor,MRW}). In the next theorem we will use this notion.

\begin{theorem}\label{subgroup}
Let $\GG$ and $\HH$ be locally compact quantum groups and assume that $\pi\colon\Linf(\HH)\to\Linf(\GG)$ and $\th\colon\Linf(\GG)\to\Linf(\GG)\vtens\Linf(\HH)$ define a projection of $\GG$ onto $\HH$. Let $U\in\M\bigl(\C_0(\hh{\GG})\tens\C_0(\HH)\bigr)$ be the bicharacter corresponding to the \wst-right quantum homomorphism $\th$. Then
\begin{enumerate}
\item\label{subgroup1} $\C_0(\HH)=\bigl\{(\omega\tens\id)U\st\omega\in\B(\Ltwo(\GG))_*\bigr\}^{\boldsymbol{-}\|\cdot\|}$ and $\HH$ is a closed quantum subgroup of $\GG$ in the sense of Woronowicz (cf.~\cite[Definition 3.2]{DKSS}).
\item\label{subgroup2} $U\in\B\bigl(\Ltwo(\GG)\tens\Ltwo(\GG)\bigr)$ is a manageable multiplicative unitary and the corresponding quantum group coincides with $\HH$. In particular $\Linf(\hh{\HH})\subset\Linf(\hh{\GG})$ and $\Delta_{\hh{\HH}}=\bigl.\Delta_{\hh{\GG}}\bigr|_{\Linf(\hh{\HH})}$ and $\HH$ is a closed quantum subgroup of $\GG$ in the sense of Vaes (\cite[Definition 3.1]{DKSS}).
\end{enumerate}
\end{theorem}

\begin{proof}
All relevant algebras and operators act on the Hilbert space $\Ltwo(\GG)$ or its tensor products. Let us therefore use a shorthand $\cH$ for this Hilbert space.

Ad \eqref{subgroup1}. We have established before the statement of the theorem that $U\in\B(\cH\tens\cH)$ is a manageable multiplicative unitary, so there is a quantum group $\KK$ corresponding to $U$. Since $U\in\M\bigl(\cK(\cH)\tens\C_0(\HH)\bigr)$ and at the same time $U\in\M\bigl(\cK(\cH)\tens\C_0(\KK)\bigr)$ generates $\C_0(\KK)$ (by \cite[Theorem 1.6]{mu} applied to $\hh{U}$), the identity map on $\C_0(\KK)$ viewed as a mapping $\C_0(\KK)\to\B(\cH)$ is a morphism from $\C_0(\KK)$ to $\C_0(\HH)$. This means that
\begin{equation}\label{KH1}
\C_0(\KK)\C_0(\HH)=\C_0(\HH).
\end{equation}

Recall now that $U$ implements the right quantum homomorphism $\th$ which is $\Delta_\HH$ (cf.~remarks before statement of the theorem). Now $\Delta_\HH$ satisfies the cancellation laws. In particular
\[
U\bigl(\C_0(\HH)\tens\I\bigr)U^*\bigl(\C_0(\HH)\tens\I\bigr)=\C_0(\HH)\tens\C_0(\HH)
\]
or, in other words,
\begin{equation}\label{toSlice}
\bigl(\C_0(\HH)\tens\I\bigr)U^*\bigl(\C_0(\HH)\tens\I\bigr)=U^*\bigl(\C_0(\HH)\tens\C_0(\HH)\bigr).
\end{equation}
Let us slice both sides of \eqref{toSlice} with functionals $\omega\in\B(\cH)_*$ on the left leg. The norm-closed linear span of the left hand side will be
\[
\bigl\{(\omega\tens\id)U\st\omega\in\B(\cH)_*\bigr\}^{\boldsymbol{-}\|\cdot\|}=\C_0(\KK),
\]
because
\[
\begin{split}
\C_0(\HH)\B(\cH)_*\C_0(\HH)&=\C_0(\HH)\cK(\cH)\B(\cH)_*\cK(\cH)\C_0(\HH)\\
&=\cK(\cH)\B(\cH)_*\cK(\cH)=\B(\cH)_*,
\end{split}
\]
since $\C_0(\HH)$ acts non-degenerately on $\cH$. Let us examine the slices of the right hand side of \eqref{toSlice}. For $a,a'\in\C_0(\HH)$ and $\omega\in\B(\cH)_*$ we have
\[
(\omega\tens\id)\bigl(U^*(a\tens{a'})\bigr)=(\omega\tens\id)\bigl(U^*(a\tens\I)\bigr)\cdot{a'}.
\]
Taking for $\omega$ a vector functional $\omega=\omega_{\xi,\eta}$ gives
\[
(\omega_{\xi,\eta}\tens\id)\bigl(U^*(a\tens\I)\bigr)=(\omega_{\xi,a\eta}\tens\id)(U^*)\in\C_0(\KK)
\]
and such elements span a dense subspace of $\C_0(\KK)$, again, because $\C_0(\HH)$ is a non-degenerate \cst-subalgebra of $\B(\cH)$.

It follows that the closed linear span of the slices of the right hand side of \eqref{toSlice} is $\C_0(\KK)\C_0(\HH)$ and equating this to the closed linear span of slices of the left hand side of \eqref{toSlice} yields
\begin{equation}\label{KH2}
\C_0(\KK)=\C_0(\KK)\C_0(\HH).
\end{equation}
Combining \eqref{KH1} and \eqref{KH2} we obtain $\C_0(\KK)=\C_0(\HH)$. The remaining part of statement \eqref{subgroup1} follows from \cite[Theorem 3.6]{DKSS}.

Ad \eqref{subgroup2}. We have already established that $U$ is a manageable multiplicative unitary. By \eqref{subgroup1} the quantum group defined by $U$ (\cite[Definition 3]{mmu}) coincides with $\HH$ (note that the right quantum homomorphism $\th=\Delta_\HH$ is implemented by $U$). 

The bicharacter $U$ belongs to $\M\bigl(\C_0(\hh{\GG})\tens\C_0(\HH)\bigr)$, so that $\C_0(\hh{\HH})\subset\M\bigl(\C_0(\hh{\GG})\bigr)\subset\Linf(\hh{\GG})$ and consequently $\Linf(\hh{\HH})\subset\Linf(\hh{\GG})$. Finally we recall that $U$ is a bicharacter, so in particular
\begin{equation}\label{DelGhU}
(\Delta_{\hh{\GG}}\tens\id)U=U_{23}U_{13}.
\end{equation}
But also $U$ is a manageable multiplicative unitary giving rise to the quantum group $\HH$, so the comultiplication on $\hh{\HH}$ satisfies
\[
(\Delta_{\hh{\HH}}\tens\id)U=U_{23}U_{13}.
\]
It follows that $\bigl.\Delta_{\hh{\GG}}\bigr|_{\Linf(\hh{\HH})}=\Delta_{\Linf(\hh{\HH})}$.
\end{proof}

\begin{remark}
By \cite[Theorem 3.5]{DKSS} we know that a closed subgroup of $\GG$ in the sense of Vaes is also a closed subgroup of $\GG$ in the sense of Woronowicz. Nevertheless we felt that the proof of statement \eqref{subgroup1} of Theorem \ref{subgroup} is of independent interest. Moreover the proof of this statement is a step in the proof of statement \eqref{subgroup2} of Theorem \ref{subgroup}.
\end{remark}

\begin{remark}
If $\GG$ is a locally compact quantum group with projection onto $\HH$ defined by $(\th,\pi)$ then it is not difficult to show that $\hh{\GG}$ is a locally compact quantum group with projection onto $\hh{\HH}$ with corresponding right quantum homomorphism 
\[
\hh{\th}\colon\Linf(\hh{\GG})\longrightarrow\Linf(\hh{\GG})\vtens\Linf(\hh{\HH})
\]
and a normal unital $*$-homomorphism
\[
\hh{\pi}\colon\Linf(\hh{\HH})\longrightarrow\Linf(\hh{\GG})
\]
defined as follows:
\begin{itemize}
\item[$\blacktriangleright$] the map $\hh{\pi}$ is the injection of $\Linf(\hh{\HH})$ into $\Linf(\hh{\GG})$ described in Theorem \ref{subgroup}\eqref{subgroup2}. Clearly $(\hh{\pi}\tens\hh{\pi})\comp\Delta_{\hh{\HH}}=\Delta_{\hh{\GG}}\comp\hh{\pi}$.
\item[$\blacktriangleright$] The normal unital and injective $*$-homomorphism $\pi\colon\Linf(\HH)\to\Linf(\GG)$ intertwining comultiplications means that $\hh{\HH}$ is a closed quantum subgroup of $\hh{\GG}$ in the sense of Vaes. In particular it is a closed quantum subgroup in the sense of Woronowicz and there exists a corresponding right quantum homomorphism $\hh{\th}\colon\Linf(\hh{\GG})\to\Linf(\hh{\GG})\vtens\Linf(\hh{\HH})$. By \cite[Theorem 3.3(3)]{DKSS} it is implemented by
\[
(\pi\tens\id)W^{\hh{\HH}}=\hh{U},
\]
where $\hh{U}$ is defined as $\flip(U^*)$. It follows that $\hh{\th}\comp\hh{\pi}=(\hh{\pi}\tens\id)\comp\Delta_{\hh{\HH}}$.
\end{itemize}
A similar result is also contained in \cite[Remarks after Definition 3.35]{RoyPhd}.
\end{remark}

\begin{example}\label{Exazb1}
Let us take for $\GG$ one of the quantum ``$az+b$'' groups introduced in \cite{azb,nazb}. The construction begins with a choice of a (not completely arbitrary) complex number $q$ and a corresponding to it multiplicative subgroup $\Gamma_q$ of $\CC\setminus\{0\}$ with a non-degenerate symmetric bicharacter $\chi\colon\Gamma_q\times\Gamma_q\to\TT$. Then $\C_0(\GG)$ turns out to be isomorphic to $\C_0(\Gamma_q\cup\{0\})\rtimes\Gamma_q$, where the action is by multiplication of complex numbers. We denote by $a$ the ``generator'' of this action, i.e.~the normal element affiliated with $\C_0(\GG)$ (\cite[Definition 1.1]{unbo}) such that for each $\gamma\in\Gamma_q$ the corresponding unitary in $\M\bigl(\C_0(\GG)\bigr)$ implementing the action of $\gamma$ on $\C_0(\Gamma_q\cup\{0\})$ is $\chi(a,\gamma)$. Then we let $b$ denote the image in $\C_0(\GG)$ of the element affiliated with $\C_0(\Gamma_q\cup\{0\})$ given by the identity function. Then $a$ and $b$ are normal elements affiliated with $\C_0(\GG)$ and
\[
ab=q^2ba\qquad\text{and}\qquad{a}b^*=q^2b^*a
\]
in an appropriate sense. Moreover $a^{-1}$ is also affiliated with $\C_0(\GG)$ and $\C_0(\GG)$ is generated by $a,a^{-1}$ and $b$ in the sense of \cite[Definition 3.1]{gen}.

The comultiplication $\Delta_\GG$ is given on generators by
\[
\Delta_\GG(a)=a\tens{a}\qquad\text{and}\qquad\Delta_\GG(b)=a\tens{b}\dot{+}b\tens\I,
\]
where $\dot{+}$ denotes the closure of a sum of unbounded operators. The locally compact group $\Gamma_q$ is a closed (quantum) subgroup of $\GG$ and the corresponding right quantum homomorphism is
\[
a\longmapsto{a\tens{a}}\qquad\text{and}\qquad{b}\longmapsto{b\tens{a}},
\]
with $a$ in the second tensor factor considered as an element affiliated with $\C_0(\Gamma_q)$. Moreover $\C_0(\Gamma_q)$ is also a subalgebra of $\M\bigl(\C_0(\GG)\bigr)$ via
\[
\pi\colon\C_0(\Gamma_q)\ni{f}\longmapsto{f(a)}\in\M\bigl(\C_0(\GG)\bigr)
\]
which can be easily seen to intertwine the appropriate comultiplication. 

All the structure described above has an extension to von Neumann algebras $\Linf(\GG)$ and $\Linf(\Gamma_q)$ and thus the quantum ``$az+b$'' group becomes a locally compact quantum group with projection onto $\Gamma_q$ (cf.~\cite[Section 6.2.3]{RoyPhd}). 
\end{example}

\section{Emergence of braided quantum group structure}\label{emer}

Let $\GG$ and $\HH$ be locally compact quantum groups and assume $\GG$ is a quantum group with projection onto $\HH$ with corresponding maps $\pi\colon\Linf(\HH)\to\Linf(\GG)$ and $\th\colon\Linf(\GG)\to\Linf(\GG)\vtens\Linf(\HH)$ (as in Definition \ref{GprojH}). Define
\[
\sN=\bigl\{x\in\Linf(\GG)\st\th(x)=x\tens\I\bigr\}.
\]
For any $x\in\sN$ we have (cf.~\eqref{rqh})
\[
(\id\tens\th)\Delta_\GG(x)=(\Delta_\GG\tens\id)\th(x)=\Delta_\GG(x)\tens\I,
\]
so that $\Delta_\GG(\sN)\subset\Linf(\GG)\vtens\sN$, i.e.~$\sN$ is a left coideal in $\Linf(\GG)$.

Consider again the bicharacter $U$ corresponding to the \wst-right quantum homomorphisms $\th$. Originally $U$ was an element of $\M\bigl(\C_0(\hh{\GG})\tens\C_0(\HH)\bigr)\subset\Linf(\hh{\GG})\vtens\Linf(\HH)$. However we embedded $\Linf(\HH)$ into $\Linf(\GG)$ (via the map $\pi$) and then found in Theorem \ref{subgroup}\eqref{subgroup2} that $\Linf(\hh{\HH})$ is a subalgebra of $\Linf(\hh{\GG})$. Thus $U$ may be considered as an element of $\Linf(\hh{\HH})\vtens\Linf(\GG)$. In order to avoid confusion let us denote $U$ considered as belonging to $\Linf(\hh{\HH})\vtens\Linf(\GG)$ by $Y$. 

Using Theorem \ref{subgroup}\eqref{subgroup2} we find that
\[
(\Delta_{\hh{\HH}}\tens\id)Y=Y_{23}Y_{13}.
\]
Moreover formula \eqref{dwa} gives
\[
(\id\tens\th)Y=Y_{12}W_{13}^\HH
\]
(recall that $U=W^\HH$). We can now use Corollary \ref{cpProp} with $\sM=\Linf(\GG)$ and $\Act=\th$ to obtain

\begin{proposition}\label{Mcp}
The von Neumann algebra $\sN$ carries an action $\act$ of $\hh{\HH}^\op$ defined by
\begin{equation}\label{actN}
\sN\ni{y}\longmapsto\act(y)=\flip\bigl(Y(\I\tens{y})Y^*\bigr)\in\Linf(\GG)\vtens\Linf(\hh{\HH}^\op)
\end{equation}
and $\Linf(\GG)$ is isomorphic to the crossed product $\sN\rtimes_\act\hh{\HH}^\op$ is such a way that $\th$ becomes the dual action of ${\hh{\hh{\HH}^{\raisebox{.27ex}{\op}}}}^\op\cong\HH$. More precisely, there exists an isomorphism $\Psi\colon\Linf(\GG)\to\sN\rtimes_\act\hh{\HH}^\op$ such that
\begin{itemize}
\item[$\blacktriangleright$] $\hh{\act}\comp\Psi=(\Psi\tens\Phi)\comp\th$,
\item[$\blacktriangleright$] $\Psi(y)=\act(y)$ for all $y\in\sN$,
\item[$\blacktriangleright$] $(\Phi\tens\Psi)Y=W_{13}^\HH$,
\end{itemize}
where $\Phi$ is the mapping defined in Subsection \ref{oppcomm}.
\end{proposition}

By a slight extension of \cite[Theorem 5.5 and Lemma 5.7]{MRW} the right quantum homomorphism $\th$ has a \emph{left} counterpart, i.e.~a normal unital $*$-homomorphism $\thL\colon\Linf(\GG)\to\Linf(\HH)\vtens\Linf(\GG)$ characterized uniquely by the equality
\[
(\th\tens\id)\comp\Delta_\GG=(\id\tens\thL)\comp\Delta_\GG.
\]

\begin{example}\label{ExthL}
Let $G$ and $H$ be locally compact groups with maps $\rh\colon{G}\to{H}$ and $\imath\colon{H}\to{G}$ as in Example \ref{ClassEx1}. Then $\thL\colon\Linf(G)\to\Linf(H)\vtens\Linf(G)\cong\Linf(H\times{G})$ is given by
\[
\begin{aligned}
\thL(f)(h,g)&=f(\imath(h)\cdot{g}),&\quad{f}\in\Linf(G),\;(h,g)\in{H}\times{G}.
\end{aligned}
\]
Recall that we are by now identifying $\Linf(H)$ with its image in $\Linf(G)$ under the map $\pi$ associated with $\rho$. This means that $\thL$ may be regarded as mapping $\Linf(G)$ to $\Linf(G)\vtens\Linf(G)$ which transforms $f\in\Linf(G)$ into the mapping
\[
G\times{G}\ni(g_1,g_2)\longmapsto{f}\bigl(\rho(g_1)g_2\bigr).
\]
\end{example}

Just as $\th$, the map $\thL$ is injective (it is implemented by a unitary). We will use this map in the next definition.

\begin{definition}
We define the von Neumann algebra $\sN\boxtimes\sN$ as the von Neumann subalgebra of $\Linf(\GG)\vtens\Linf(\GG)$ generated by $\sN\tens\I$ and $\thL(\sN)$. We will call this von Neumann algebra the \emph{braided tensor product} of $\sN$ with itself.
\end{definition}

The terminology ``braided tensor product'' is not the only possibility. The paper \cite{MRW1} dealing with analogous objects on \cst-algebra level calls them ``twisted tensor products''. In \cite{WorBraid} the term ``crossed product'' is used. Our terminology is related to the notion of a braided group \cite{majid,WorBraid}, but we will not attempt to develop here the full theory of such braided tensor products of von Neumann algebras. The important feature of $\sN\boxtimes\sN$ is that it is defined as the algebra generated by two copies of $\sN$, namely $\sN\tens\I$ and $\thL(\sN)$, which do not necessarily commute (as they would in a regular tensor product). This is analogous to the more precise \cst-algebraic notions discussed in e.g.~\cite{majid0,majid} in the purely algebraic context and in \cite{WorBraid,MRW1} in the \cst-algebraic context.

Our aim in the remainder of this section is to show that $\sN$ carries a natural comultiplication with values not in $\sN\vtens\sN$, but in $\sN\boxtimes\sN$. We begin with the discussion of this phenomenon in the classical case.

\begin{example}\label{ClassEx2}
Let, as in Example \ref{ClassEx1}, $G$ and $H$ be locally compact groups with continuous homomorphisms $\rh\colon{G}\to{H}$ and $\imath\colon{H}\to{G}$ such that
\[
\rh\comp\imath=\id_{H}
\]
and associated maps $\pi\colon\Linf(H)\to\Linf(G)$ and $\th\colon\Linf(G)\to\Linf(G)\vtens\Linf(H)$. The fixed point algebra
\[
\sN=\bigl\{x\in\Linf(\GG)\st\th(x)=x\tens\I\bigr\}
\]
consists of functions on $\GG$ which are constant on left cosets of $H$ in $G$ (since $\imath$ is injective we regard $H$ as a subgroup of $G$). This algebra is isomorphic to the algebra of functions on $K=\ker{\rh}$.

As $G$ is isomorphic to $K\rtimes{H}$, we can write any element $g\in{G}$ uniquely as a product $g=kh$ (cf.~Section \ref{intro}). Using this we can define a map $\varLambda\colon\sN\vtens\sN\to\Linf(G)\vtens\Linf(G)$ by
\[
\begin{aligned}
\varLambda(\mathcal{F})(kh,k'h')&=\mathcal{F}(k,hk'),&\quad\mathcal{F}\in\sN\vtens\sN,\;k,k'\in{K},\;h,h'\in{H}.
\end{aligned}
\]
One easily finds that this map is an isomorphism onto its image. Moreover
\[
\begin{aligned}
\varLambda(f_1\tens\I)&=f_1\tens\I,&\quad{f_1}\in\sN
\end{aligned}
\]
and
\[
\begin{aligned}
\varLambda(\I\tens{f_2})&=\thL(f_2),&\quad{f_2}\in\sN,
\end{aligned}
\]
where in this context, $\thL$ maps $f\in\Linf(G)$ to a function $H\times{G}\ni(h,g)\mapsto{f}(hg)$ (cf.~Example \ref{ExthL}). It follows that $\varLambda$ is an isomorphism of $\sN\vtens\sN$ onto $\sN\boxtimes\sN$. 

Moreover, $\varLambda$ has another crucial property. Recall that $\sN$ is isomorphic to $\Linf(K)$, so it carries its own comultiplication $\Delta_K$. It is easily checked that
\begin{equation}\label{DKDG}
\varLambda\comp\Delta_K=\bigl.\Delta_G\bigr|_\sN.
\end{equation}
\end{example}

We will proceed to develop the theory of locally compact quantum groups with projection along the lines suggested in Example \ref{ClassEx2}. The crucial point is that in the fully non-commutative setting the map $\varLambda$ described above does not necessarily exist (the algebra $\sN\vtens\sN$ may no longer be isomorphic to $\sN\boxtimes\sN$), so that formula \eqref{DKDG} does not make sense. Nevertheless, as we will see in Theorem \ref{glowne} below, its right hand side continues to define a map $\sN\to\sN\boxtimes\sN$. 

It turns out that the quantum group $\GG$ may be described within the framework of \emph{extensions} (or \emph{short exact sequences}) of locally compact quantum groups from \cite{VaesVainerman} if only if $\sN$ commutes with $\Linf(\HH)$ inside $\Linf(\GG)$. In this case $\sN\boxtimes\sN$ is isomorphic to $\sN\vtens\sN$ and $\sN$ with comultiplication from $\Linf(\GG)$ is an algebra of functions on a quantum group and it gives rise to the ``normal subgroup'' entering the extension data (\cite{KaspSol2}).

Let, as before, $\GG$ and $\HH$ be locally compact quantum groups and let $(\pi,\th)$ define a projection of $\GG$ onto $\HH$. Define a unitary $F\in\Linf(\hh{\GG})\vtens\Linf(\GG)$ by
\[
W^\GG=FU,
\]
i.e.~$F={W^\GG}U^*$. This unitary is the \emph{braided multiplicative unitary} of \cite[Section 6]{RoyPhd}.

We have
\[
(\id\tens\th)F=(\id\tens\th)(W^\GG{U^*})=W^\GG_{12}U_{13}U_{13}^*U_{12}^*=F_{12},
\]
so $F$ actually belongs to $\Linf(\hh{\GG})\vtens\sN$. Therefore it is clear that
\[
\cN=\bigl\{(\omega\tens\id)F\st\omega\in\B(\Ltwo(\GG))_*\bigr\}
\]
is a subspace of $\sN$. It turns out that a stronger result is true:

\begin{proposition}\label{cNdense}
The subspace $\cN$ is ultra-weakly dense in $\sN$.
\end{proposition}

For the proofs of Proposition \ref{cNdense} and Theorem \ref{glowne} we need several facts about $F$ (versions of these can be found in \cite[Proof of Theorem 6.7]{RoyPhd}).

\begin{lemma}
We have
\begin{subequations}
\begin{align}
U_{13}F_{23}U_{13}^*&=U_{12}^*F_{23}U_{12},\label{UFU}\\
(\id\tens\Delta_\GG)F&=F_{12}U_{12}F_{13}U_{12}^*.\label{DelF}
\end{align}
\end{subequations}
\end{lemma}

\begin{proof}
The information we will use is that $U$ and $W^\GG$ are multiplicative unitaries:
\[
\begin{split}
W_{23}^\GG{W_{12}^\GG}{W_{23}^\GG}^*&=W_{12}^\GG{W_{13}^\GG},\\
U_{23}U_{12}U_{23}^*&=U_{12}U_{13}
\end{split}
\]
and \eqref{dwa} which we rewrite as
\begin{equation}\label{dwadwa}
W_{23}^\GG{U_{12}}{W_{23}^\GG}^*=U_{12}U_{13}.,
\end{equation}
Using this and the pentagon equation for $W^\GG$ we compute
\[
\begin{split}
F_{12}U_{12}F_{13}U_{12}^*&=W^\GG_{12}U_{12}^*U_{12}W^\GG_{13}U_{13}^*U_{12}^*\\
&=W^\GG_{12}W^\GG_{13}U_{13}^*U_{12}^*\\
&=W^\GG_{23}W^\GG_{12}{W^\GG_{23}}^*W^\GG_{23}U_{12}^*{W^\GG_{23}}^*\\
&=W^\GG_{23}W^\GG_{12}U_{12}^*{W^\GG_{23}}^*\\
&=W^\GG_{23}F_{12}{W^\GG_{23}}^*
\end{split}
\]
which proves \eqref{DelF}.

To prove \eqref{UFU} note that, by the pentagon equation for $U$, \eqref{dwadwa} says that
\[
F_{23}U_{12}U_{13}F_{23}^*=F_{23}U_{23}U_{12}U_{23}^*F_{23}^*=W^\GG_{23}U_{12}{W^\GG_{23}}^*=U_{12}U_{13}
\]
and that means that $F_{23}$ commutes with $U_{12}U_{13}$. Therefore
\[
F_{23}^*(U_{12}U_{13})F_{23}(U_{12}U_{13})^*=(U_{12}U_{13})(U_{12}U_{13})^*=\I
\]
which is \eqref{UFU}.
\end{proof}

\begin{proof}[Proof of Proposition \ref{cNdense}]
Applying $\id\tens\omega\tens\id$ to both sides of \eqref{UFU} with all $\omega\in\B\bigl(\Ltwo(\GG)\bigr)_*$ and using the definition \eqref{actN} of $\act$ (recall that $U$ and $Y$ are the same element) we find that  the (slices of the) left hand side form the set $\flip\bigl(\act(\cN)\bigr)$. The slices of the right hand side all belong to $\overline{\Linf(\hh{\HH})\tens\cN}^{\:\!\text{\tiny\rm{w}}}\!\!$. Thus $\cN$ is $\act$-invariant in the following sense
\[
\act(\cN)\subset\overline{\cN\tens\Linf(\hh{\HH})}^{\:\!\text{\tiny\rm{w}}}.
\]

For the second step consider a slice of $W^\GG$ with a vector functional:
\[
\begin{split}
(\omega_{\xi,\eta}\tens\id)W^\GG&=(\omega_{\xi,\eta}\tens\id)(FU)\\
&=(\omega_{\xi,\eta}\tens\id)\biggl(F\sum_i\ket{e_i}\!\bra{e_i}U\biggr)\\
&=\sum_i\bigl((\omega_{\xi,e_i}\tens\id)F\bigr)\bigl((\omega_{e_i,\eta}\tens\id)U\bigr),
\end{split}
\]
where $\{e_i\}$ is an orthonormal basis of $\Ltwo(\GG)$. This shows that $\Linf(\GG)$ is the ultra-weakly closed linear span of products $ny$ with $n\in\cN$ and $y\in\Linf(\HH)$. Now using the isomorphism $\Psi$ from Proposition \ref{Mcp} we see that $\sN\rtimes_\act\hh{\HH}^\op$ is the ultra-weakly closed linear span of products of the form $\act(y)(\I\tens{x'})$ with $y\in\cN$ and $x'\in\Linf(\HH)'$. Thus, by Corollary \ref{invN}, the subspace $\cN$ is ultra-weakly dense in $\sN$.
\end{proof}

\begin{theorem}\label{glowne}
We have $\Delta_\GG(\sN)\subset\sN\boxtimes\sN$.
\end{theorem}

\begin{proof}
By \cite[Theorem 5.5]{MRW} we have
\begin{equation}\label{key1}
(\id\tens\thL)(W^\GG)=U_{12}W^\GG_{13}.
\end{equation}
Now we note that on the subalgebra $\Linf(\HH)$ of $\Linf(\GG)$ the map $\th$ coincides with $\Delta_\HH$ (i.e.~$\bigl.\th\bigr|_{\Linf(\HH)}=\bigl.\Delta_\HH$). It follows that for any $y\in\Linf(\HH)$, $\omega\in\B\bigl(\Ltwo(\GG)\bigr)_*$ and $x=(\omega\tens\id)\Delta_\HH(y)$ we have
\[
\begin{split}
\thL(x)&=(\omega\tens\id)\bigl((\id\tens\thL)\Delta_\HH(y)\bigr)\\
&=(\omega\tens\id)\bigl((\id\tens\thL)\Delta_\GG(y)\bigr)\\
&=(\omega\tens\id)\bigl((\th\tens\id)\Delta_\GG(y)\bigr)\\
&=(\omega\tens\id)\bigl((\th\tens\id)\Delta_\HH(y)\bigr)\\
&=(\omega\tens\id)\bigl((\Delta_\HH\tens\id)\Delta_\HH(y)\bigr)\\
&=(\omega\tens\id)\bigl((\id\tens\Delta_\HH)\Delta_\HH(y)\bigr)=\Delta_\HH(x).
\end{split}
\]
Since such elements $x$ are ultra-weakly dense in $\Linf(\HH)$ we find that on $\Linf(\HH)$ we have $\thL=\Delta_\HH$ (and so it is also equal to $\th$). In particular
\begin{equation}\label{key2}
(\id\tens\thL)U=U_{12}U_{13}.
\end{equation}

Using \eqref{key1} and \eqref{key2} we obtain
\[
\begin{split}
(\id\tens\thL)F&=(\id\tens\thL)(W^\GG{U^*})\\
&=U_{12}W^\GG_{13}U_{13}^*U_{12}^*\\
&=U_{12}F_{13}U_{12}^*,
\end{split}
\]
so by \eqref{DelF}
\[
(\id\tens\Delta_\GG)F=F_{12}\bigl((\id\tens\thL)F\bigr),
\]
which shows that $\Delta_\GG(\cN)\subset\sN\boxtimes\sN$ and ends the proof.
\end{proof}

\begin{example}
Let us return to the quantum ``$az+b$'' groups mentioned already in Example \ref{Exazb1}. The \emph{reduced bicharacter}, i.e.~the image in $\M\bigl(\C_0(\hh{\GG})\tens\C_0(\GG)\bigr)$ of the multiplicative unitary operator giving rise to $\GG$ (see \cite[Section 4]{SolWor}) is explicitly given in \cite{azb,nazb} in the form
\[
W^\GG=\FF(\hh{b}\tens{b})\chi(\hh{a}\tens\I,\I\tens{a}),
\]
where $\FF$ is a certain special function, $a,b$ are the elements described in Example \ref{Exazb1} and $\hh{a},\hh{b}$ are analogous elements affiliated with $\C_0(\hh{\GG})$. This, of course is precisely the decomposition $W^\GG=FU$ discussed above. The algebra $\sN$ spanned by slices of $F$ is in this case equal to
\[
\bigl\{f(b)\st{f}\in\Linf(\Gamma_q\cup\{0\})\bigr\}
\]
which is isomorphic to $\Linf(\Gamma_q\cup\{0\})$ (see Example \ref{Exazb1}). The map $\th$ is in this case the (von Neumann algebra extension of the) action of $\Gamma_q$ on $\GG$ used in \cite[Section 3]{exa} to select elements of the \cst-algebra of function on the quantum homogeneous space $\GG/\Gamma_q$.
\end{example}

\end{document}